\documentclass[10pt]{amsart}
\usepackage{amsmath}
\usepackage{amssymb}
\usepackage{amscd}
\usepackage{url}
\usepackage{pstricks}
\usepackage{pstricks-add}

\newcommand{\ol}[1]{\overline{#1}}

\numberwithin{equation}{section}

\theoremstyle{plain}
\newtheorem{theorem}[equation]{Theorem}
\newtheorem{lemma}[equation]{Lemma}

\newtheorem{proposition}[equation]{Proposition}

\newtheorem{conjecture}[equation]{Conjecture}

\theoremstyle{definition}
\newtheorem{example}[equation]{Example}
\newtheorem{remark}[equation]{Remark}

\newtheorem{definition}[equation]{Definition}

\numberwithin{equation}{section}
\newtheorem*{ack}{Acknowledgements}

\def\Z{\mathbb Z}
\def\Q{\mathbb Q}
\def\C{\mathbb C}
\def\F{\Z_2}
\newcommand{\intfrac}[2]{\genfrac{\lfloor}{\rfloor}{}{1}{#1}{#2}}
\def\p{\partial}
\def\sm{\setminus}
\def\CP{\C P}
\DeclareMathOperator{\CFK}{\it CFK}

\newcommand{\spinc}{\ifmmode{{\mathfrak s}}\else{${\mathfrak s}$\ }\fi}
\newcommand{\spinct}{\ifmmode{{\mathfrak t}}\else{${\mathfrak t}$\ }\fi}
\def\Spc{Spin$^c$}
\DeclareMathOperator\St{St}
\DeclareMathOperator\Vrt{Vert}

\def\staircasefig{%
\begin{pspicture}(-3.5,-3.5)(3.5,3.5)
\rput(-2,-2){\psscalebox{0.6}{%
\psgrid[subgriddiv=1,gridcolor=gray](-1,-1)(-1,-1)(7,7)
\psline[linewidth=1.6pt,linecolor=gray](-1,0)(7,0)
\psline[linewidth=1.6pt,linecolor=gray](0,-1)(0,7)
\psline[linewidth=2pt,linecolor=black,ArrowInside=->](6,1)(6,0)
\psline[linewidth=2pt,linecolor=black,ArrowInside=->](4,2)(4,1)
\psline[linewidth=2pt,linecolor=black,ArrowInside=->](2,4)(2,2)
\psline[linewidth=2pt,linecolor=black,ArrowInside=->](1,6)(1,4)
\psline[linewidth=2pt,linecolor=black,ArrowInside=->](1,6)(0,6)
\psline[linewidth=2pt,linecolor=black,ArrowInside=->](2,4)(1,4)
\psline[linewidth=2pt,linecolor=black,ArrowInside=->](4,2)(2,2)
\psline[linewidth=2pt,linecolor=black,ArrowInside=->](6,1)(4,1)

\pscircle[fillstyle=solid,fillcolor=black,linecolor=black](6,0){0.15}
\pscircle[fillstyle=solid,fillcolor=black,linecolor=black](1,4){0.15}
\pscircle[fillstyle=solid,fillcolor=black,linecolor=black](2,2){0.15}
\pscircle[fillstyle=solid,fillcolor=black,linecolor=black](4,1){0.15}
\pscircle[fillstyle=solid,fillcolor=black,linecolor=black](0,6){0.15}
\pscircle[fillstyle=solid,fillcolor=gray,linecolor=black](6,1){0.15}
\pscircle[fillstyle=solid,fillcolor=gray,linecolor=black](1,6){0.15}
\pscircle[fillstyle=solid,fillcolor=gray,linecolor=black](4,2){0.15}
\pscircle[fillstyle=solid,fillcolor=gray,linecolor=black](2,4){0.15}
}}
\end{pspicture}}
\def\infinitestaircasefig{%
\begin{pspicture}(-3.5,-3.5)(3.5,3.5)
\rput(-2,-2){\psscalebox{0.6}{%
\psgrid[subgriddiv=1,gridcolor=gray](-1,-1)(-1,-1)(7,7)
\psline[linewidth=1.6pt,linecolor=gray](-1,0)(7,0)
\psline[linewidth=1.6pt,linecolor=gray](0,-1)(0,7)
\psline[linewidth=2pt,linecolor=black,ArrowInside=->](6,1)(6,0)
\psline[linewidth=2pt,linecolor=black,ArrowInside=->](4,2)(4,1)
\psline[linewidth=2pt,linecolor=black,ArrowInside=->](2,4)(2,2)
\psline[linewidth=2pt,linecolor=black,ArrowInside=->](1,6)(1,4)
\psline[linewidth=2pt,linecolor=black,ArrowInside=->](1,6)(0,6)
\psline[linewidth=2pt,linecolor=black,ArrowInside=->](2,4)(1,4)
\psline[linewidth=2pt,linecolor=black,ArrowInside=->](4,2)(2,2)
\psline[linewidth=2pt,linecolor=black,ArrowInside=->](6,1)(4,1)
\pscircle[fillstyle=solid,fillcolor=black,linecolor=black](6,0){0.15}
\pscircle[fillstyle=solid,fillcolor=black,linecolor=black](1,4){0.15}
\pscircle[fillstyle=solid,fillcolor=black,linecolor=black](2,2){0.15}
\pscircle[fillstyle=solid,fillcolor=black,linecolor=black](4,1){0.15}
\pscircle[fillstyle=solid,fillcolor=black,linecolor=black](0,6){0.15}
\pscircle[fillstyle=solid,fillcolor=gray,linecolor=black](6,1){0.15}
\pscircle[fillstyle=solid,fillcolor=gray,linecolor=black](1,6){0.15}
\pscircle[fillstyle=solid,fillcolor=gray,linecolor=black](4,2){0.15}
\pscircle[fillstyle=solid,fillcolor=gray,linecolor=black](2,4){0.15}

\psline[linewidth=2pt,linecolor=black,ArrowInside=->](7,2)(7,1)
\psline[linewidth=2pt,linecolor=black,ArrowInside=->](5,3)(5,2)
\psline[linewidth=2pt,linecolor=black,ArrowInside=->](3,5)(3,3)
\psline[linewidth=2pt,linecolor=black,ArrowInside=->](2,7)(2,5)
\psline[linewidth=2pt,linecolor=black,ArrowInside=->](2,7)(1,7)
\psline[linewidth=2pt,linecolor=black,ArrowInside=->](3,5)(2,5)
\psline[linewidth=2pt,linecolor=black,ArrowInside=->](5,3)(3,3)
\psline[linewidth=2pt,linecolor=black,ArrowInside=->](7,2)(5,2)
\pscircle[fillstyle=solid,fillcolor=black,linecolor=black](7,1){0.15}
\pscircle[fillstyle=solid,fillcolor=black,linecolor=black](2,5){0.15}
\pscircle[fillstyle=solid,fillcolor=black,linecolor=black](3,3){0.15}
\pscircle[fillstyle=solid,fillcolor=black,linecolor=black](5,2){0.15}
\pscircle[fillstyle=solid,fillcolor=black,linecolor=black](1,7){0.15}
\pscircle[fillstyle=solid,fillcolor=gray,linecolor=black](7,2){0.15}
\pscircle[fillstyle=solid,fillcolor=gray,linecolor=black](2,7){0.15}
\pscircle[fillstyle=solid,fillcolor=gray,linecolor=black](5,3){0.15}
\pscircle[fillstyle=solid,fillcolor=gray,linecolor=black](3,5){0.15}

\psline[linewidth=2pt,linecolor=black,ArrowInside=->](6,4)(6,3)
\psline[linewidth=2pt,linecolor=black,ArrowInside=->](4,6)(4,4)
\psline[linewidth=2pt,linecolor=black,ArrowInside=->](3,7)(3,6)
\psline[linewidth=2pt,linecolor=black,ArrowInside=->](4,6)(3,6)
\psline[linewidth=2pt,linecolor=black,ArrowInside=->](6,4)(4,4)
\psline[linewidth=2pt,linecolor=black,ArrowInside=->](7,3)(6,3)
\pscircle[fillstyle=solid,fillcolor=black,linecolor=black](3,6){0.15}
\pscircle[fillstyle=solid,fillcolor=black,linecolor=black](4,4){0.15}
\pscircle[fillstyle=solid,fillcolor=black,linecolor=black](6,3){0.15}
\pscircle[fillstyle=solid,fillcolor=gray,linecolor=black](6,4){0.15}
\pscircle[fillstyle=solid,fillcolor=gray,linecolor=black](4,6){0.15}

\psline[linewidth=2pt,linecolor=black,ArrowInside=->](7,5)(7,4)
\psline[linewidth=2pt,linecolor=black,ArrowInside=->](5,7)(5,5)
\psline[linewidth=2pt,linecolor=black,ArrowInside=->](5,7)(4,7)
\psline[linewidth=2pt,linecolor=black,ArrowInside=->](7,5)(5,5)
\psline[linewidth=2pt,linecolor=black,ArrowInside=->](7,3)(6,3)
\pscircle[fillstyle=solid,fillcolor=black,linecolor=black](4,7){0.15}
\pscircle[fillstyle=solid,fillcolor=black,linecolor=black](5,5){0.15}
\pscircle[fillstyle=solid,fillcolor=black,linecolor=black](7,4){0.15}
\pscircle[fillstyle=solid,fillcolor=gray,linecolor=black](7,5){0.15}
\pscircle[fillstyle=solid,fillcolor=gray,linecolor=black](5,7){0.15}

\psline[linewidth=2pt,linecolor=black,ArrowInside=->](6,7)(6,6)
\psline[linewidth=2pt,linecolor=black,ArrowInside=->](7,6)(6,6)
\pscircle[fillstyle=solid,fillcolor=black,linecolor=black](6,6){0.15}

\psline[linewidth=2pt,linecolor=black,ArrowInside=->](5,0)(5,-1)
\psline[linewidth=2pt,linecolor=black,ArrowInside=->](3,1)(3,0)
\psline[linewidth=2pt,linecolor=black,ArrowInside=->](1,3)(1,1)
\psline[linewidth=2pt,linecolor=black,ArrowInside=->](0,5)(0,3)
\psline[linewidth=2pt,linecolor=black,ArrowInside=->](0,5)(-1,5)
\psline[linewidth=2pt,linecolor=black,ArrowInside=->](1,3)(0,3)
\psline[linewidth=2pt,linecolor=black,ArrowInside=->](3,1)(1,1)
\psline[linewidth=2pt,linecolor=black,ArrowInside=->](5,0)(3,0)
\pscircle[fillstyle=solid,fillcolor=black,linecolor=black](5,-1){0.15}
\pscircle[fillstyle=solid,fillcolor=black,linecolor=black](0,3){0.15}
\pscircle[fillstyle=solid,fillcolor=black,linecolor=black](1,1){0.15}
\pscircle[fillstyle=solid,fillcolor=black,linecolor=black](3,0){0.15}
\pscircle[fillstyle=solid,fillcolor=black,linecolor=black](-1,5){0.15}
\pscircle[fillstyle=solid,fillcolor=gray,linecolor=black](5,0){0.15}
\pscircle[fillstyle=solid,fillcolor=gray,linecolor=black](0,5){0.15}
\pscircle[fillstyle=solid,fillcolor=gray,linecolor=black](3,1){0.15}
\pscircle[fillstyle=solid,fillcolor=gray,linecolor=black](1,3){0.15}

\psline[linewidth=2pt,linecolor=black,ArrowInside=->](2,0)(2,-1)
\psline[linewidth=2pt,linecolor=black,ArrowInside=->](0,2)(0,0)
\psline[linewidth=2pt,linecolor=black,ArrowInside=->](-1,4)(-1,2)
\psline[linewidth=2pt,linecolor=black,ArrowInside=->](0,2)(-1,2)
\psline[linewidth=2pt,linecolor=black,ArrowInside=->](2,0)(0,0)
\psline[linewidth=2pt,linecolor=black,ArrowInside=->](4,-1)(2,-1)
\pscircle[fillstyle=solid,fillcolor=black,linecolor=black](-1,2){0.15}
\pscircle[fillstyle=solid,fillcolor=black,linecolor=black](0,0){0.15}
\pscircle[fillstyle=solid,fillcolor=black,linecolor=black](2,-1){0.15}
\pscircle[fillstyle=solid,fillcolor=gray,linecolor=black](4,-1){0.15}
\pscircle[fillstyle=solid,fillcolor=gray,linecolor=black](-1,4){0.15}
\pscircle[fillstyle=solid,fillcolor=gray,linecolor=black](2,0){0.15}
\pscircle[fillstyle=solid,fillcolor=gray,linecolor=black](0,2){0.15}

\psline[linewidth=2pt,linecolor=black,ArrowInside=->](-1,1)(-1,-1)
\psline[linewidth=2pt,linecolor=black,ArrowInside=->](1,-1)(-1,-1)

\pscircle[fillstyle=solid,fillcolor=black,linecolor=black](-1,-1){0.15}
\pscircle[fillstyle=solid,fillcolor=gray,linecolor=black](1,-1){0.15}
\pscircle[fillstyle=solid,fillcolor=gray,linecolor=black](-1,1){0.15}
}}
\end{pspicture}}

\def\functionJm{%
\begin{pspicture}(-5,-5)(5,5)
\rput(0,0){\psscalebox{0.6}{%
\psgrid[subgriddiv=1,gridcolor=gray](0,0)(-7,-7)(7,7)
\psline[linewidth=1.2pt,linecolor=gray](-1,0)(7,0)
\psline[linewidth=1.2pt,linecolor=gray](0,-1)(0,7)
\psline[linewidth=2pt,linecolor=black](6,0)(6,1)(4,1)(4,2)(2,2)(2,4)(1,4)(1,6)(0,6)
\psline[linestyle=dashed,linecolor=blue,linewidth=4pt](0,5)(0,3)
\psline[linestyle=dashed,linecolor=blue,linewidth=4pt](0,2)(0,0)
\psline[linestyle=dashed,linecolor=blue,linewidth=4pt](0,-2)(0,-3)
\psline[linestyle=dashed,linecolor=blue,linewidth=4pt](0,-5)(0,-6)
\psline[linestyle=dashed,linecolor=blue,linewidth=4pt](0,7)(0,6)

\psline[linestyle=dashed, linecolor=gray](6,1)(0,-5)
\psline[linestyle=dashed, linecolor=gray](4,2)(0,-2)
\psline[linestyle=dashed,linewidth= 1.4pt, linecolor=black](4,1)(0,-3)
\psline[linestyle=dashed, linewidth= 1.4pt,linecolor=black](2,2)(0,0)
\psline[linestyle=dashed, linecolor=gray](1,6)(0,5)
\psline[linestyle=dashed, linecolor=gray](2,4)(0,2)
\psline[linestyle=dashed, linewidth= 1.4pt, linecolor=black](6,0)(0,-6)
\psline[linestyle=dashed,linewidth= 1.4pt, linecolor=black](1,4)(0,3)

\pscircle[fillstyle=solid,fillcolor=black,linecolor=black](6,0){0.15}
\pscircle[fillstyle=solid,fillcolor=black,linecolor=black](1,4){0.15}
\pscircle[fillstyle=solid,fillcolor=black,linecolor=black](2,2){0.15}
\pscircle[fillstyle=solid,fillcolor=black,linecolor=black](4,1){0.15}
\pscircle[fillstyle=solid,fillcolor=black,linecolor=black](0,6){0.15}
\pscircle[fillstyle=solid,fillcolor=gray,linecolor=black](6,1){0.15}
\pscircle[fillstyle=solid,fillcolor=gray,linecolor=black](1,6){0.15}
\pscircle[fillstyle=solid,fillcolor=gray,linecolor=black](4,2){0.15}
\pscircle[fillstyle=solid,fillcolor=gray,linecolor=black](2,4){0.15}

\psline[linestyle=solid,linecolor=purple,linewidth=4pt](0,6)(0,5)
\psline[linestyle=solid,linecolor=purple,linewidth=4pt](0,3)(0,2)
\psline[linestyle=solid,linecolor=purple,linewidth=4pt](0,0)(0,-2)
\psline[linestyle=solid,linecolor=purple,linewidth=4pt](0,-3)(0,-5)
\psline[linestyle=solid,linecolor=purple,linewidth=4pt](0,-6)(0,-7)
\pscircle[linestyle=solid,fillcolor=white,fillstyle=solid](0,5){0.15}
\pscircle[linestyle=solid,fillcolor=white,fillstyle=solid](0,3){0.15}
\pscircle[linestyle=solid,fillcolor=white,fillstyle=solid](0,2){0.15}
\pscircle[linestyle=solid,fillcolor=white,fillstyle=solid](0,0){0.15}
\pscircle[linestyle=solid,fillcolor=white,fillstyle=solid](0,-2){0.15}
\pscircle[linestyle=solid,fillcolor=white,fillstyle=solid](0,-3){0.15}
\pscircle[linestyle=solid,fillcolor=white,fillstyle=solid](0,-5){0.15}
\pscircle[linestyle=solid,fillcolor=white,fillstyle=solid](0,-6){0.15}
\pscircle[linestyle=solid,fillcolor=white,fillstyle=solid](0,6){0.15}
}}
\end{pspicture}}

\title[Rational cuspidal curves]{Heegaard Floer homology and  rational cuspidal curves}
\author{Maciej Borodzik}
\address{Institute of Mathematics, University of Warsaw, ul. Banacha 2,
02-097 Warsaw, Poland}
\email{mcboro@mimuw.edu.pl}
\thanks{The first author was supported by  Polish OPUS grant No 2012/05/B/ST1/03195}
\thanks{The second author was supported by
 National Science Foundation   Grant  1007196.}

\author{Charles Livingston}
\address{Department of Mathematics, Indiana University, Bloomington, IN 47405}
\email{livingst@indiana.edu}

\makeatletter
\@namedef{subjclassname@2010}{\textup{2010} Mathematics Subject Classification}
\makeatother

\begin{document}
\begin{abstract}
We apply the methods of Heegaard Floer homology to identify topological properties of complex curves in $\CP^2$.  As one application, we resolve an open conjecture that constrains the Alexander polynomial of the link of the singular point of the curve in the case that there is exactly one singular point, having connected link, and the curve is of genus 0.  Generalizations apply in the case of multiple singular points.

\end{abstract}
\maketitle

\section{Introduction}

We consider irreducible algebraic curves $C\subset \CP^2$. Such a curve has a finite set of singular points, $\{z_i\}_{i=1}^{n}$; 
a neighborhood of each intersects $C$ in a cone on a link $L_i \subset S^3$. 
A  fundamental question  asks what possible configurations of  links $\{L_i\}$ arise in this way.  
In this generality the problem is fairly intractable  and research has focused on a restricted case, 
in which each $L_i$ is connected, and thus a knot $K_i$, and  $C$ is a rational curve, meaning that there is a rational 
surjective map $\CP^1 \to C$.   
Such a curve is called   {\it rational cuspidal}.  Being rational cuspidal is equivalent to $C$ being homeomorphic to $S^2$. 

Our results apply in the case of multiple singular points, but the following statement gives an indication of the nature of the results and their
consequences.\vskip.1in

\begin{theorem}\label{thm:first}
Suppose that $C$ is a rational cuspidal curve of degree $d$ with one singular point, a cone on the knot $K$, and the Alexander polynomial of $K$ is expanded at $t=1$ to be $\Delta_K(t) = 1 +\frac{(d-1)(d-2)}{2}(t-1) + (t-1)^2\sum_l k_l t^l$.  Then for all $j, 0 \le j \le d-3$, $k_{d(d-j-3)} = (j-1)(j-2)/2$.
\end{theorem}
\vskip.1in 

There are three facets to the work here:

\begin{enumerate}
\item We begin   with a basic observation that a neighborhood $Y$ of $C$  is built from the 4--ball by attaching a 2--handle along the knot $K =\#K_i$ with framing $d^2$, where $d$ is the degree of the curve.  Thus, its boundary, $S^3_{d^2}(K)$, bounds the rational homology ball $\CP^2\sm Y$.  From this, it follows that the Heegaard Floer correction term satisfies $d(S^3_{d^2}(K), \spinc_{m}) = 0$   if $d | m$, for properly enumerated \Spc structures $\spinc_m$.

\item  Because each $K_i$ is an algebraic knot (in particular an  $L$--space knot), the Heegaard Floer complex $\CFK^\infty(S^3, K_i)$ is determined by the Alexander polynomial of $K_i$, and thus the complex $\CFK^\infty(S^3, K)$ and the $d$--invariants are also determined by the Alexander polynomials of the $K_i$.   

\item  The constraints that arise on the Alexander polynomials, although initially appearing quite intricate, can be reinterpreted in compact form 
using semigroups of singular points. In this way, we can relate these constraints to well-known conjectures.

\end{enumerate}

\subsection{The conjecture of Fern\'andez de Bobadilla, Luengo, Melle-Hernandez and N\'emethi}
In \cite{FLMN04} the following conjecture was proposed. It was also verified for all known examples of rational cuspidal curves.

\begin{conjecture}[\cite{FLMN04}]\label{conj:one}
Suppose that the rational cuspidal curve $C$ of degree $d$ has critical points $z_1,\dots,z_n$. Let $K_1,\dots,K_n$ be the corresponding links
of singular points and let $\Delta_1,\dots,\Delta_n$ be their Alexander polynomials. Let $\Delta=\Delta_1\cdot\ldots\cdot\Delta_n$, expanded as
\[\Delta(t)=1+\frac{(d-1)(d-2)}{2}(t-1) +(t-1)^2\sum_{j=0}^{2g-2}k_lt^l.\]
Then for any $j=0,\dots,d-3$,  $k_{d(d-j-3)}\le (j+1)(j+2)/2$, with equality for $n=1$.
\end{conjecture}
We remark that the case $n=1$ of the conjecture is  Theorem~\ref{thm:first}. We will prove this result in Section~\ref{S:proof}.
Later we will  also prove  an alternative generalization of Theorem~\ref{thm:first} 
for the case $n>1$, stated as Theorem~\ref{thm:main},  which is the main result of the present article.  
The advantage of this formulation  over the original conjecture lies in the fact that it gives precise values of the coefficients $k_{d(d-j-3)}$.
 Theorem~\ref{cor:main} provides an equivalent statement of Theorem~\ref{thm:main}.

\begin{ack}
The authors are grateful to Matt Hedden, Jen Hom and Andr\'as N\'emethi for fruitful discussions. The first author wants to thank Indiana University for
hospitality. 
\end{ack}

\section{Background: Algebraic Geometry and Rational Cuspidal Curves}
In this section we will present some of the general theory of rational cuspidal curves.   Section~\ref{sec:singpoints} includes
basic information about singular points of plane curves. In Section~\ref{subsecsemi} we discuss the semigroup
of a singular point and its connections to the Alexander polynomial of the link. We shall use   results from this section later in
the article to simplify the equalities that we obtain. In Section~\ref{subseccusp} we describe results from \cite{FLMN04}
to give some flavor of the theory. In Section~\ref{subsecrough} we provide a rough sketch of some methods used to
study rational cuspidal curves. We refer to \cite{Moe08} for an excellent and fairly up-to-date survey of results on rational cuspidal curves.

\subsection{Singular points and algebraic curves}\label{sec:singpoints}
For a general introduction and references to this subsection, we refer to \cite{BK, GLS}, or
to \cite[Section 10]{Singpoints} for a more topological approach.   In this article we will be considering algebraic curves embedded in $\CP^2$.  Thus we will use the word  \emph{curve}  to refer to  a zero set of an irreducible 
homogeneous polynomial $F$ of degree $d$. The \emph{degree} of the curve is the degree of the corresponding
polynomial.

Let $C$ be a curve. A point $z\in C$ is called \emph{singular} if the gradient of $F$ vanishes at $z$. Singular points of irreducible
curves in $\CP^2$ are always isolated.
Given a singular point and a sufficiently
small ball $B\subset\CP^2$ around $z$, we call $K=C\cap\p B$ the \emph{link} of the singular point. The singular point is called \emph{cuspidal}
or \emph{unibranched} if $K$ is a knot, that is a link with one component, or equivalently, if there is an analytic map $\psi$ from a disk in $\C$ onto $C\cap B$.

Unless specified otherwise,   all singular points are assumed to be cuspidal.

Two unibranched singular points are called \emph{topologically equivalent} if the links of these singular points are isotopic;   see for instance
\cite[Definition I.3.30]{GLS} for more details.   A  unibranched singular point is topologically equivalent to one for which the local parametrization $\psi$ is given in local coordinates $(x,y)$ on $B$
  by $t\mapsto (x(t),y(t))$, where
$x(t)=t^p$, $y(t)=t^{q_1}+\ldots+t^{q_n}$ for some
positive integers $p,q_1,\ldots,q_n$ satisfying $p<q_1<q_2<\ldots<q_n$. Furthermore, if we set
$D_i=\gcd(p,q_1,\ldots,q_i)$, then $D_i$ does not divide $q_{i+1}$ and $D_n=1$. The sequence $(p;q_1,\ldots,q_n)$ is called the
\emph{characteristic sequence} of the singular point and $p$ is called the \emph{multiplicity}. Sometimes $n$ is referred to
as the \emph{number of Puiseux pairs}, a notion which comes from an alternative way of encoding the sequence $(p;q_1,\ldots,q_n)$.   We will say that a singular point is of {\it type} $(p;q_1,\ldots,q_n)$ if it has a presentation of this sort in local coefficients.

The link of a singular point with a characteristic sequence $(p;q_1,\ldots,q_n)$ is an $(n-1)$--fold iterate of a torus knot $T(p',q')$, where $p'=p/D_1$
and $q'=q_1/D_1$; see for example \cite[Sections 8.3 and 8.5]{BK} or  \cite[Chapter 5.2]{Wall04}. 
In particular, if $n=1$, the link is a torus knot $T(p,q_1)$. In all cases, the genus of the link is equal to $\mu/2=\delta$,
where $\mu$ is the Milnor number and $\delta$ is the so-called $\delta$--invariant of the singular point, see \cite[page 205]{GLS}, or \cite[Section 10]{Singpoints}.
The genus is also equal to half the degree of the Alexander polynomial of the link of the singular point.
The Milnor number can be computed from the following formula, see \cite[Remark 10.10]{Singpoints}:
\[\mu=(p-1)(q_1-1)+\sum_{i=2}^n (D_i-1)(q_i-q_{i-1}).\]

Suppose that $C$ is a degree $d$ curve  with singular points $z_1,\ldots,z_n$ (and $L_1,\ldots,L_n$ are their links). 
The genus formula, due to Serre
(see \cite[Property 10.4]{Singpoints})  states that the genus of $C$ is equal to
\[g(C)=\frac12(d-1)(d-2)-\sum_{i=1}^n\delta_i.\]
If all the critical points are cuspidal, we have $\delta_i=g(L_i)$, so the above formula can be written as
\begin{equation}\label{eq:gofC}
g(C)=\frac12(d-1)(d-2)-\sum_{i=1}^n g(L_i).
\end{equation}
In particular, $C$ is rational cuspidal (that is, it is a homeomorphic image of a sphere) if and only $\sum g(L_i)=\frac12(d-1)(d-2)$.

\subsection{Semigroup of a singular point}\label{subsecsemi}  The notion of the semigroup associated to 
a singular point is a central notion in the subject, although in the present work we use only the language of semigroups, not the algebraic aspects.  
We refer to \cite[Chapter 4]{Wall04} or \cite[page 214]{GLS} for details and proofs. Suppose that $z$ is a cuspidal singular point
of a curve $C$ and $B$ is a sufficiently small ball around $z$. Let $\psi(t)=(x(t),y(t))$ be a local parametrization of $C\cap B$ near $z$; see
Section~\ref{sec:singpoints}.
For any polynomial $G(x,y)$ we look at the order at $0$ of an analytic map $t\mapsto G(x(t),y(t))\in\C$. 
Let $S$ be the set integers, which can be realized as the order for some $G$. 
Then $S$ is clearly a semigroup of $\Z_{\ge 0}$. We call it the \emph{semigroup
of the singular point}. The semigroup can be computed from the characteristic sequence, for example for a sequence $(p;q_1)$, $S$ is generated by $p$ and $q_1$.
The \emph{gap sequence}, $G:=\Z_{\ge 0}\sm S$, has precisely $\mu/2$ elements and the largest one is $\mu-1$, where $\mu$ is the Milnor number.

We now assume that $K$ is the link of the singular point $z$.
Explicit computations of the Alexander polynomial of $K$ show
that it is of the form
\begin{equation}\label{eq:alex1}
\Delta_K(t)=\sum_{i=0}^{2m}(-1)^it^{n_i},
\end{equation}
where $n_i$ form an increasing sequence with $n_0 = 0$ and $n_{2m} = 2g$, twice the genus of $K$.

Expanding $t^{n_{2i}}-t^{n_{2i-1}}$ as $(t-1)(t^{n_{2i}-1}+t^{n_{2i}-2}+\ldots+t^{n_{2i-1}})$ yields
\begin{equation}\label{eq:alex2}
\Delta_K(t)=1+(t-1)\sum_{j=1}^k t^{g_j},
\end{equation}
for some finite sequence $0<g_1<\ldots<g_k$. We have the following result (see \cite[Exercise 5.7.7]{Wall04}).
\begin{lemma}\label{lem:gaplem}
The sequence $g_1,\ldots,g_k$ is the gap sequence of the semigroup of the singular point. In particular $k=\#G=\mu/2$, where $\mu$ is the Milnor number,
so $\#G$ is the genus.
\end{lemma}
Writing $t^{g_j}$ as $(t-1)(t^{g_j-1}+t^{g_j-2}+\ldots+t+1)+1$ in \eqref{eq:alex2} yields the following formula
\begin{equation}\label{eq:alex3}
\Delta_K(t)=1+(t-1)g(K)+(t-1)^2\sum_{j=0}^{\mu-2}k_jt^{j},
\end{equation}
where $k_j=\#\{m>j\colon m\not\in S\}$.

We shall use the following definition.
\begin{definition}\label{def:iofg}
For any finite increasing sequence of positive integers $G $, we  define  
\begin{equation}\label{eq:gapfunction}
I_G(m)  =  \#\{ k \in G\cup\Z_{<0} \colon k \ge m\},
\end{equation}
where $\Z_{<0}$ is the set of the negative integers. We shall call $I_G$ the \emph{gap function}, because in most applications $G$ will be a gap
sequence of some semigroup.
\end{definition}
\begin{remark}\label{rem:kj}
We point out that for $j=0,\ldots,\mu-2$, we have $I_G(j+1)=k_j$, where the $k_j$ are as in \eqref{eq:alex3}.
\end{remark}
\begin{example}
Consider the knot $ T(3,7)$.  Its  Alexander polynomial is
\begin{align*}
\frac{(t^{21}-1)(t-1)}{(t^3-1)(t^7-1)}   =&\    1-t+t^3-t^4+t^6-t^8+t^9-t^{11}+t^{12} \\
   =&\ 1+(t-1)(t+t^2+t^4+t^5+t^8+t^{11}) \\=&
\ 1+6(t-1)+\\
+(t-1)^2&\left(6+5t+4t^2+4t^3+3t^4+2t^5+2t^6+2t^7+t^8+t^9+t^{10}\right).
\end{align*}
The semigroup is $(0,3,6,7,9,10,12,13,14,\dots)$. The gap sequence is $1,2,4,5,8,11$.
\end{example}

\begin{remark}\label{rem:gapofk}
The passage from \eqref{eq:alex1} through \eqref{eq:alex2} to \eqref{eq:alex3} is just an algebraic manipulation, and thus it applies to  any knot  whose
Alexander polynomial has form \eqref{eq:alex1}.  In particular, according to~\cite[Theorem 1.2]{OzSz05} it applies    to  any $L$--space knot. 
In this setting we will also call the sequence $g_1,\dots,g_k$ 
the \emph{gap sequence} of the knot and denote it by $G_K$; we will write $I_K(m)$ for the gap function relative to $G_K$.
 Even though the complement  $\Z_{\ge 0}\sm G_K$ is not always a semigroup, we still 
have $\#G_K=\frac12\deg\Delta_K$. This property follows  immediately from the symmetry of the Alexander
polynomial.
\end{remark}

\subsection{Rational cuspidal curves with one cusp}\label{subseccusp}
The classification of rational cuspidal curves is a challenging old problem, with some conjectures (like the Coolidge--Nagata conjecture \cite{Cool28,Naga60}) 
remaining open for many decades. The classification of curves with a unique critical 
point is far from being accomplished; 
the special case when the unique singular point has only one Puiseux term (its link is a torus knot) is complete~\cite{FLMN04}, but even in this basic case, 
the proof is quite difficult.

To give some indication  of the situation, consider two families of rational cuspidal curves. The first one, written in projective coordinates on $\CP^2$
as $x^d+y^{d-1}z=0$ for $d>1$,  the other one is $(zy-x^2)^{d/2}-xy^{d-1}=0$ for $d$ even and $d>1$.  These are of degree $d$. Both families have a unique
singular point, in the first case it is of type $(d-1;d)$, in the second of type $(d/2;2d-1)$. In both cases, the Milnor number is $(d-1)(d-2)$, so the curves are rational.
An explicit parametrization can be easily given as well.

There also  exist 
more complicated examples.  For instance,    Orevkov~\cite{Orev02} constructed rational cuspidal curves of degree $\phi_j$ having a single singular point
of type $(\phi_{j-2};\phi_{j+2})$, where $j$ is odd and $j>5$. Here the $\phi_j$ are the Fibonacci numbers, $\phi_0=0$, $\phi_1=1$, $\phi_{j+2}=\phi_{j+1}+\phi_j$.
As an example, there exists a rational cuspidal curve of degree $13$ with a single singular point of type $(5;34)$. Orevkov's construction is inductive and
by no means trivial. Another family found by Orevkov 
are rational cuspidal curves of degree $\phi_{j-1}^2-1$ having a single singular point of type $(\phi_{j-2}^2;\phi_j^2)$,
for $j>5$, odd.

The main result of \cite{FLMN04} is that apart of these four families of rational cuspidal curves, there are only two sporadic curves with
a unique singular point having one Puiseux pair, one of degree $8$, 
the other   of degree $16$. 

\subsection{Constraints on rational cuspidal curves.}\label{subsecrough}
Here we review some constraints for rational cuspidal curves. We refer to \cite{Moe08} for more details and references. 
The article \cite{FLMN04} shows how these constraints can be used in practice. The fundamental constraint is  given by \eqref{eq:gofC}.
Next, Matsuoka and Sakai \cite{MaSa89}
proved that if $(p_1;q_{11},\ldots,q_{1k_1})$, \dots,$(p_n;q_{n1},\ldots,q_{nk_n})$ are the only  singular points occurring on a rational cuspidal curve of degree $d$
with $p_1\ge \ldots\ge p_n$, then $p_1>d/3$. Later, Orevkov \cite{Orev02} improved this to  $\alpha(p_1+1)+1/\sqrt{5}>d$, where $\alpha=(3+\sqrt{5})/2\sim 2.61$
and showed that this inequality is asymptotically optimal (it is related to the curves described   in Section~\ref{sec:singpoints}).
Both proofs use  very deep algebro-geometric tools. We reprove the result of \cite{MaSa89} in Proposition~\ref{prop:matsa} below.

Another obstruction comes from the semicontinuity of the spectrum, a concept that arises from Hodge Theory. 
Even a rough definition of the spectrum of a singular point is beyond the scope of this article. We refer to
\cite[Chapter 14]{AVG} for a definition  of the spectrum  and to \cite{FLMN04} for illustrations of its use. We point out that recently (see \cite{BN})
a tight relation has been drawn between the spectrum of a singular point and the Tristram--Levine signatures of its link. 
In general, semicontinuity of the spectrum is a very strong tool, but
it is also very difficult to apply.

Using tools from algebraic geometry, such as the Hodge Index Theorem, Tono in \cite{Tono05} proved that any rational cuspidal curve can have at most eight
singular points. An old conjecture is that a rational cuspidal curve can have at most $4$ singular points; see \cite{Pion}
for a precise statement.

In \cite{FLMN06} a completely new approach was proposed, motivated   by a  conjecture on Seiberg--Witten
invariants of links of surface singularities made by N\'emethi and Nicolaescu; see \cite{NeNi05}. Specifically,  
Conjecture~\ref{conj:one} in the present article arises from these considerations. 
Another reference for the general conjecture on Seiberg--Witten invariants is~\cite{Nem08}.
 
\section{Topology, algebraic topology, and \Spc{} structures}

Let $C\subset\CP^2$ be a rational cuspidal curve. Let $d$ be its degree and $z_1,\ldots,z_n$ be its singular points.
We let $Y$ be a closed manifold regular  neighborhood of $C$, let $M = \partial Y$, and $W = \overline {\CP^2 - Y}$.
\subsection{ Topological descriptions of $Y$ and $ M$}
The neighborhood $Y$ of $C$ can be built  in three steps.  First, disk neighborhoods of the $z_i$ are selected.  Then neighborhoods of $N-1$ embedded arcs on $C$ are adjoined, yielding a 4--ball.  Finally, the remainder of $C$ is a disk, so its neighborhood forms a 2--handle attached to the 4--ball.  Thus, $Y$ is a 4--ball  
with a 2--handle attached.  The attaching curve is easily seen to be $K = \# K_i$.   Finally, since the self-intersection of $C$ is $d^2$, the framing of the attaching map is $d^2$.   In particular, $M = S^3_{d^2}(K)$.

\smallskip
One quickly computes that  $ H_2(\CP^2, C) = \Z_d$, and $H_4(\CP^2, C) = \Z$, with the remaining homology groups 0.   Using excision, we see that  the groups $H_i( W, M)$ are the same.   Via Lefschetz duality and the universal coefficient theorem we find that $H_0(W) = \Z$, $H_1(W) = \Z_d$ and all the other groups are 0. Finally, the long exact sequence of the pair $(W, M)$ yields 
$$0 \to H_2(W, M) \to H_1(M) \to H_1(W) \to 0$$ which in this case is $$0 \to \Z_d \to \Z_{d^2} \to \Z_d \to 0.$$  

This is realized geometrically by letting the generator of $H_2(W,M)$ be ${\it H} \cap W$, where $\it H\subset\CP^2$ is a generic line.  
Its boundary is algebraically $d$ copies of the meridian of the attaching curve $K$ in the 2--handle decomposition of $Y$.

Taking duals we see that the map $H^2(W) \to H^2(M)$, which maps $\Z_d \to \Z_{d^2}$, takes the canonical generator to $d$ times the dual to the meridian in $M = S^3_{d^2}(K)$.

\subsection{ \Spc{} structures}

For any space $X$ there  is a  transitive action of $H^2(X)$ on  \Spc($X$).  Thus, $W$ has $d$ \Spc  structures and $M$ has $d^2$ such structures.  

Since $\CP^2$ has a \Spc  structure  with first Chern class a dual to the class of the line, 
its restriction to $W$ is a structure whose restriction to $M$ has first Chern class equal to $d$ times the dual to the meridian.  

For a cohomology class $z \in H^2(X)$ and a \Spc   structure $\spinc$, one has $c_1(z\cdot \spinc)- c_1(\spinc) = 2z$.  Thus for each $k\in\Z$, there is a \Spc  structure on $M$ which extends to $W$ having first Chern class of the form $d +2kd$.  Notice that for $d$ odd, all $md \in \Z_{d^2}$ for $m\in\Z$ occur as first Chern classes of   \Spc  structures that extend over $W$, but for $d$ even, only elements of the form $md$ with $m$ odd occur.  (Thus, for $d$ even, there are $d$ extending structures, but only $d/2$ first Chern classes that occur.) 

According to~\cite[Section 3.4]{OzSz04}, the \Spc  structures on $M$ have an enumeration $\spinc_m$, for $m \in  [ -d^2/2 , d^2/2]$, which can be defined via the manifold $Y$.  Specifically, $\spinc_m$ is defined to be the restriction to $M$ of the \Spc  structure on $Y$, $\spinct_m$, with the property that $\left< c_1(\spinct_m), C\right> + d^2 = 2m$. We point out that if $d$ is even, $\spinc_{d^2/2}$ and $\spinc_{-d^2/2}$ denote the same structure; compare Remark~\ref{rem:same} below.

It now follows from our previous observations that the structures $\spinc_m$ that extend to $W$ are those   with $m = kd$ for some integer $k$,  $-d/2 \le k \le d/2$ if $d$ is odd.  If $d$ is even, then those that extend have $m = kd/2$ for some odd $k$,  $-d\le k \le d$.  For future reference, we summarize this with the following lemma.

\begin{lemma}\label{extendinglemma}  If $W^4  = \overline{  \CP^2 - Y}$  where $Y$ is a neighborhood of a rational cuspidal curve $C$ of degree $d$ (as constructed
above), then the \Spc  structure $\spinc_m$ on $\partial  W^4$ extends to $W^4$ if     $m = kd$ for some integer $k$,  $-d/2 \le k \le d/2$ if $d$ is odd.  If $d$ is even, then those that extend have $m = kd/2$ for some odd $k$,  $-d\le k \le d$.   Here $\spinc_m$ is the \Spc  structure on $\partial W$  which extends to a structure $\spinct$ on $Y$ satisfying  $\left< c_1(\spinct_m), C\right> + d^2 = 2m$.

\end{lemma}

\section{Heegaard Floer theory}\label{S:HFK}

Heegaard Floer theory~\cite{OzSz03} associates to a 3--manifold $M$ with  \Spc  structure $\spinc$, a filtered, graded chain complex $CF^\infty(M,\spinc)$ over the field $\F$.  A fundamental invariant of the pair $(M, \spinc)$, the {\it correction term} or $d$--invariant, $d(M,\spinc) \in \Q$, is determined by $CF^\infty(M,\spinc)$.  The manifold $M$ is called an {\it $L$--space} if certain associated homology groups are of rank one~\cite{OzSz05}.

A knot  $K$ in $M$ provides a second filtration on $CF^\infty(M,\spinc)$~\cite{OzSz03}.  In particular, for $K   \subset S^3$  there is a bifiltered graded chain complex $\CFK^\infty(K)$ over the field $\F$.  It is known that for algebraic knots  the complex is determined by the Alexander polynomial of $K$.  More generally, this holds for any knot upon which some surgery yields an $L$--space; these knots are called $L$--space knots.

The Heegaard Floer invariants of surgery on $K$, 
in particular the $d$--invariants of $S^3_{q}(K)$, are determined by this complex, and for  $q>2(\textrm{genus}(K))$  
the computation of $d(S^3_q(K), \spinc)$ from $CFK^\infty(K)$ is particularly simple.  In this section we  will illustrate the general theory, leaving the details to references such 
as~\cite{HHN12,HLR12}.

\subsection{$\CFK^\infty(K)$ for $K$ an algebraic knot}
Figure~\ref{fig:staircasefig}  is a schematic illustration of a finite complex over $\F$.  Each dot represents a generator and the arrows indicate boundary maps.  Abstractly it is of the form $0 \to \F^4 \to \F^5 \to 0$ with homology $\F$.   The complex is bifiltered with the horizontal and vertical coordinates representing the filtrations levels of the generators.  We will refer to the two filtrations levels as the $(i,j)$--filtrations levels.   The complex has an absolute grading which is not indicated in the diagram; the generator at filtration level $(0,6)$ has grading 0 and the boundary map lowers the
grading by 1.  Thus, there are five generators at grading level 0 and four at grading level one.  We call the first set of generators type {\bf A} and the second type {\bf B}. 

We will refer to a complex such as this as a {\it staircase complex} of length $n$, $\St(v)$, where $v$ is a $(n-1)$--tuple of positive integers designating the length of the segments starting at the top left and moving to the bottom right in alternating right and downward steps.  Furthermore we require
that the top left vertex lies on the vertical axis and the bottom right vertex lies on the horizontal axis. 
Thus, the illustration is of $\St(1,2,1,2,2,1,2,1)$.
The absolute grading
of $\St(v)$ is defined by setting the grading of the top left generator to be equal to $0$ and the boundary map to lower the grading by $1$.

 The vertices of $\St(K)$ will be denoted $\Vrt(St(K))$. We shall write $\Vrt_A(\St(K))$ to denote the set of  type {\bf A} vertices and write $\Vrt_B(\St(K))$ for the set of  vertices of type {\bf B}.

If $K$ is a knot admitting an $L$--space surgery, in particular an algebraic knot (see \cite{Hed09}),
then it has Alexander polynomial of the form $\Delta_K(t) = \sum_{i=0}^{2m} (-1)t^{n_i}$.
To such a knot we associate a staircase complex, $\St(K)=\St(n_{i+1} - n_i)$, where $i$ runs from 0 to $2m-1$.  As an example, the torus knot $T(3,7)$ has Alexander polynomial $1 -t +t^3 -t^4 + t^6 -t^8 +t^9 - t^{11} +t^{12}$.  The corresponding staircase complex is  $\St(1,2,1,2,2,1,2,1)$.

\begin{figure}[t]
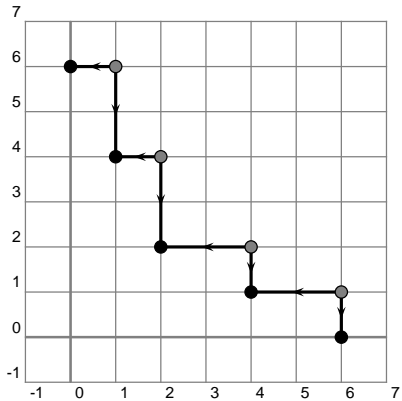

\staircasefig
\caption{The staircase complex $\St(K)$ for the torus knot $T(3,7)$. }\label{fig:staircasefig}
\end{figure}

Given any finitely generated bifiltered complex $S$, one can form a larger complex $S \otimes \F[U,U^{-1}]$, with differentials defined by $\partial( x \otimes U^i )=( \partial x) \otimes U^i$.  It is graded by $gr (x \otimes U^k ) = gr(x) - 2k$.  Similarly, if $x$ is at filtration level $(i,j)$, then $x \otimes U^i$ is at filtration level $(i - k, j - k)$.  If $K$ admits an $L$--space surgery, then $\St(K)\otimes\F[U,U^{-1}]$ is isomorphic to $\CFK^\infty(K)$.  
Figure~\ref{infinitestaircasefig} illustrates a portion of  $\St(T(3,7)) \otimes \F[U,U^{-1}]$; that is, a portion of the Heegaard Floer complex
$\CFK^\infty(T(3,7))$.

\begin{figure}
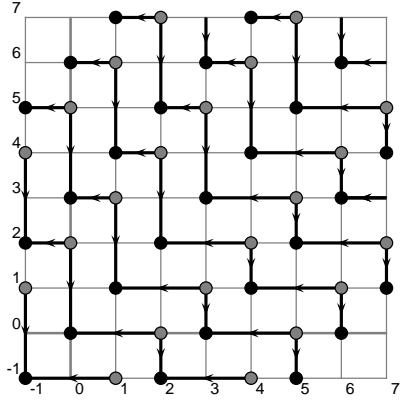

\infinitestaircasefig
\caption{A portion of $\CFK^\infty(T(3,7))$.  }\label{infinitestaircasefig}
\end{figure}

\subsection{$d$--invariants from $\CFK^\infty(K)$.}  We will not present the general definition of the $d$--invariant of a 3--manifold 
with \Spc  structure; details can be found in~\cite{OzSz03}.  
However, in the case that a 3--manifold is of the form $S^3_{q}(K)$ where $q \ge 2($genus($K$)), 
there is a simple algorithm (originating from \cite[Section 4]{OzSz04}, we use the approach of \cite{HHN12,HLR12})
to determine this invariant from $\CFK^\infty(K)$.

If $m$ satisfies $-d/2 \le m \le d/2$, one can form the quotient complex $$\CFK^\infty(K) /\CFK^\infty(K)\{i <0, j <m\}.$$
We let $d_m$ denote the least grading in which this complex has a nontrivial homology class, say $[z]$, where $[z]$ must satisfy 
the added constraint that for all $i>0$,  $[z] = U^i [ z_i]$ for some homology class $[z_i]$ of grading $d_m + 2i$.

In~\cite[Theorem 4.4]{OzSz04}, we find the following result.

\begin{theorem}\label{dinvthm} For the \Spc  structure $\spinc_m$, $d(S^3_{q}(K), \spinc_m) = d_m + \frac{ (-2m +q)^2 - q}{4q}$.

\end{theorem} 

\subsection{From staircase complexes to the $d$--invariants}

Let us now define a distance function for a staircase complex by the formula
\[
J_K(m)=\min_{(v_1, v_2)\in \Vrt(\St(K))} \max(v_1,v_2-m),
\]
where $v_1,v_2$ are coordinates of the vertex $v$. Observe that
the minimum can always be   taken with respect to the set of vertices of type {\bf A}.
 The function $J_K(m)$ represents the greatest $r$ such that the region $\{i\le 0, j \le m\}$ intersects $\St(K) \otimes U^{r}$ nontrivially.  
It is immediately clear that $J_K(m)$ is a non-increasing function. It is also immediate  that for $m\ge g$ we have $J_K(m)=0$.

\begin{figure}
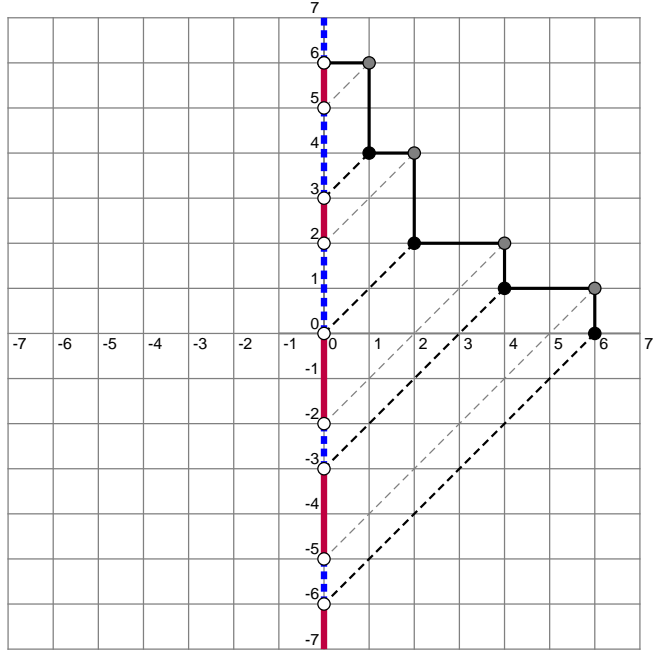

\functionJm
\caption{The function $J(m)$ for the knot $T(3,7)$. 
When $(0,m)$ lies on the dashed vertical intervals, the function $J(m)$
is constant; when it is on solid vertical intervals the function $J(m)$ is decreasing. The dashed lines connecting vertices to points on the vertical axis
indicate how the ends of dashed and solid intervals are constructed. }\label{fig:two}
\end{figure}

For the sake of the next lemma we define $n_{-1}=-\infty$.
\begin{lemma}\label{growththm}
Suppose $m\le g$. We have $J_K(m+1)-J_K(m)=-1$ if $n_{2i-1}-g\le m< n_{2i}-g$ for some $i$, and $J_K(m+1)=J_K(m)$ otherwise.
\end{lemma}
\begin{proof}
The proof is purely combinatorial. 
We order the type {\bf A} vertices of $\St(K)$ so that the first coordinate is increasing, and we denote these vertices  $v_0,\ldots,v_k$. For example, 
for $\St(T(3,7))$ as depicted on Figure~\ref{fig:staircasefig}, we have $v_0=(0,6)$, $v_1=(1,4)$, $v_2=(2,2)$,
$v_3=(4,1)$ and $v_4=(6,0)$. We denote by $(v_{i1},v_{i2})$ the coordinates of the vertex $v_i$.

A verification of the two following facts is straightforward:
\begin{equation}\label{eq:twosimple}
\begin{split}
\max(v_{i1},v_{i2}-m)&=v_{i1} \textrm{ if and only if   $m\ge v_{i1}-v_{i2}$}\\
\max(v_{i1},v_{i2}-m)&\ge \max(v_{i-1,1},v_{i-1,2}-m) \textrm{ if and only if  $m\le v_{i1}-v_{i-1,2}$}.
\end{split}
\end{equation}
By the definition of the staircase complex we also  have  $v_{i1}-v_{i2}=n_{2i}-g$ and $v_{i1}-v_{i-1,2}=n_{2i-1}-g$.
The second equation of \eqref{eq:twosimple} yields
\[
J_K(m)=\max(v_{i1},v_{i2}-m) \text{ if and only if } m\in[n_{2i-1},n_{2i+1}].
\]
Then the first equation of \eqref{eq:twosimple} allows to compute the difference $J_K(m+1)-J_K(m)$.
\end{proof}

The relationship between $J_K$ and the $d$--invariant is given by the next result.
\begin{proposition}\label{prop:dofalg}
Let $K$ be an algebraic knot, let $q>2g(K)$, and let $m\in[-q/2,q/2]$ be an integer. Then  \[d(S^3_q(K),\mathfrak{s}_m)=\frac{(-2m+q)^2-q}{4q}-2J(m).\]
\end{proposition}

\begin{proof}  Denote by $S_i$ the subcomplex $\St(K) \otimes U^i$ in $\CFK^\infty(K)$.  The result depends on understanding the homology of the image of $S_i$ in 
 $\CFK^\infty(K) /  \CFK^\infty(K)\{i<0, j<m\}$. Because of the added constraint (see the paragraph before Theorem~\ref{dinvthm}),
we only have to look at the homology classes supported on images of the type {\bf A} vertices.
Notice that if $i > J_K(m)$, then at least one of the type {\bf A} vertices is in  
 $\CFK^\infty(K)\{i<0, j<m\}$.  But all the type {\bf A}  vertices are homologous in $S_i$, and since these generate $H_0(S_i)$, the homology of the image 
in the quotient is 0.  On the other hand, if $i\le J_K(m)$, then none of the vertices of $S_i$ are in  $\CFK^\infty(K)\{i<0, j<m\}$ and thus the homology of $S_i$  survives in the quotient.  
 
 It follows that the least grading of a nontrivial class in the quotient arises from the $U^{J_K(m)}$ translate of one of type {\bf A} vertices of $S_0=\St(K)$.  
Since $U$ lowers grading by 2, the grading is  $-2J_K(m)$.  The result follows by applying the shift described in Theorem~\ref{dinvthm}.
\end{proof}

\begin{remark}\label{rem:same}
Notice that in the case that $q$ is even, the integer values $m = -q/2$ and $m=q/2$ are both in the given interval.  One easily checks that 
Proposition~\ref{prop:dofalg} yields the same value at these two endpoints.
\end{remark}

We now relate the $J$ function to the semigroup of the singular point. Let $I_K$ be the gap function as in Definition~\ref{def:iofg} and Remark~\ref{rem:gapofk}.

\begin{proposition}\label{jequalithm} If $K$ is the link of an algebraic singular point, then for $-g \le m \le g$  $J_K(m) = I_{K}(m + g)$.
\end{proposition}

\begin{proof} In Section~\ref{subsecsemi} we described the gap sequence in terms of the exponents $n_i$.  
It follows immediately that the growth properties of $I_K(m+g)$ are identical to those of $J_K(m)$, as described in Lemma~\ref{growththm}.  Furthermore, since the largest element in the gap sequence is $2g-1$, we have $I_K(2g)=J_K(g)=0$.
\end{proof}

\subsection{Proof of Theorem~\ref{thm:first}}\label{S:proof}

According to Lemma~\ref{extendinglemma}, the \Spc  structures on $S^3_{d^2}(K)$ that extend to the complement $W$ of a neighborhood of $C$ are precisely those $\spinc_m$ where $m = kd$ for some $k$, where $-d/2 \le k \le d/2$; here $k\in\Z$ if $d$ odd, and $k\in\Z+\frac12$ if $d$ is even.  
Since $W$ is a rational homology sphere, by \cite[Proposition 9.9]{OzSz03}
the associated $d$--invariants are 0, so by Proposition~\ref{prop:dofalg}, 
letting $q = d^2$ and $m= kd$, we have $$2J_K(kd) = \frac{(-2kd +d^2)^2  - d^2}{4d^2}. $$  
By Proposition~\ref{jequalithm} we can replace this with 
\[8I_{G_K}(kd + g) = (d - 2k -1)(d - 2k+1).\]
Now $g=d(\frac{d-3}{2})+1$, so by substituting $j=k+\frac{d-3}{2}$ we obtain
\[8I_{K}(jd+1)=4(d-j+1)(d-j+2)\]
and $j\in[-3/2,\ldots,d-3/2]$ is an integer regardless of the parity of $d$. The proof is accomplished by recalling that $k_{jd}=I_K(jd+1)$, see Remark~\ref{rem:kj}.

\section{Constraints on general rational cuspidal curves}

\subsection{Products of staircase complexes and the $d$--invariants} 
 In the case that there is more than one cusp, the previous approach continues to apply, except the knot $K$ is now a connected sum of algebraic knots. 

For the connected sum $K = \#K_i$,  the complex $\CFK^\infty(K)$ is the tensor product of the $\CFK^\infty(K_i)$.  To analyze this, we consider the tensor product 
of the staircase complexes $\St(K_i)$.    Although this is not a staircase complex, the homology is still $\F$, supported at grading level 0. For the tensor product
we shall denote by $\Vrt(\St(K_1)\otimes\ldots\otimes\St(K_n))$ the set of vertices of the corresponding complex. These are of the form $v_1+\ldots+v_n$,
where $v_j\in\Vrt(K_j)$, $j=1,\ldots,n$.

 Any element of the form $a_{1q_1}\otimes a_{2q_2}\otimes \cdots \otimes a_{nq_n}$ represents a generator of the homology of the tensor product, 
where the $a_{iq_i}$ are vertices of type {\bf A} taken from each $\St(K_i)$.  Furthermore, if the translated subcomplex  $\text{St}(K) \otimes U^i \subset \text{St}(K) \otimes \F[U, U^{-1}]$ intersects $\CFK^\infty(K)\{i<0, j<m\}$ nontrivially, then the intersection contains one of these generators.  Thus, the previous argument applies to prove the following.

\begin{proposition}\label{prop:d-surg}
Let $q>2g-1$, where $g=g(K)$ and $m\in[-q/2,q/2]$. Then we have
\[d(S^3_q(K),\mathfrak{s}_m)=-2J_K(m)+\frac{(-2m+q)^2-q}{4q},\]
where $J_K(m)$ is the minimum of $\max(\alpha,\beta-m)$ over all elements of form $a_{1q_1}\otimes a_{2q_2}\otimes\ldots\otimes a_{nq_n}$,
where $(\alpha,\beta)$, is the filtration level of the corresponding element.
\end{proposition}

 Since the $d$--invariants vanish for all \Spc  structures that extend to $W$, we have:
 
 \begin{theorem}\label{Jofcomposite}  If $C$ is a rational cuspidal curve of degree $d$ with singular points $K_i$ and $K = \#K_i$, then for all $k$ in the 
range $[-d/2,d/2]$, with $k\in\Z$ for $d$ odd and $k\in\Z+\frac12$ for $d$ even:
 \[  J_K(kd) = \frac{ (d-2k -1)(d-2k+1)}{8}.\]
 
 \end{theorem}
  \begin{proof} We have from the vanishing of the    $d$--invariants, $d( S^3_{d^2}(K), \spinc_m)$  (for $m = kd$) the condition 
  $$  J_K(m) = \frac{ (-2m + d^2)^2 - d^2}{8d^2}.$$
 The result then follows by substituting   $m = kd$ and performing algebraic simplifications.
\end{proof}

\subsection{Restatement in terms of $I_{K_i}(m)$.}  We now wish to restate Theorem~\ref{Jofcomposite} in terms of the coefficients of the Alexander polynomial, properly expanded.  As before, for the gap sequence for the knot $K_i$,  denoted $G_{K_i}$, let
\[
I_i(s)=\#\{k\ge s\colon k\in G_{K_i}\cup\Z_{<0}\}.
\]
For two functions $I,I'\colon\Z\to\Z$ bounded below we define the following operation
\begin{equation}\label{eq:convolution}
I\diamond I'(s)=\min_{m\in\Z} I(m)+I'(s-m).
\end{equation}
As pointed out to us by  Krzysztof Oleszkiewicz, in real analysis this operation is sometimes called the \emph{infimum convolution}.

The following is   the main result of this article.
\begin{theorem}\label{thm:main}
Let $C$ be a rational cuspidal curve of degree $d$.  Let $I_1,\dots,I_n$ be the gap functions associated to each singular point on $C$.   
Then for any $j\in\{-1,0,\ldots,d-2\}$ we have
\[I_1\diamond I_2\diamond \ldots \diamond I_n(jd+1)=\frac{(j-d+1)(j-d+2)}{2}.\]
\end{theorem}
\begin{remark} $\ $

\begin{itemize}
\item For $j=-1$, the left hand side is $d(d-1)/2=d-1+(d-1)(d-2)/2$.  The meaning of the equality is that $\sum\# G_j=(d-1)(d-2)/2$ which
follows from \eqref{eq:gofC} and Lemma~\ref{lem:gaplem}. Thus, the case $j=-1$
does not provide any new constraints. Likewise, for $j=d-2$ both sides are equal to $0$.
\item We refer to Section~\ref{sec:reform} for a reformulation of Theorem~\ref{thm:main}.
\item We do not know
if Theorem~\ref{thm:main} settles Conjecture~\ref{conj:one} for $n>1$.  The passage between the two formulations appears to be more complicated; see~\cite[Proposition 7.1.3]{NeRo12}  and the example in Section~\ref{sec:certain}.
\end{itemize}
\end{remark}

Theorem~\ref{thm:main} is an immediate  consequence of the arguments in Section~\ref{S:proof} together with the following proposition.

\begin{proposition}\label{prop:ikjk}
As in \eqref{eq:convolution}, let $I_K$ be given by $I_1\diamond\ldots\diamond I_n$, for the 
gap functions $I_1,\ldots,I_n$. Then $J_K(m)=I_K(m+g)$.
\end{proposition}

\begin{proof}
The proof follows by  induction over $n$. For $n=1$, the statement is equivalent to Proposition~\ref{jequalithm}. Suppose we have proved it for $n-1$.
Let $K'=K_1\#\ldots\# K_{n-1}$ and let $J_{K'}(m)$ be the corresponding $J$ function.
Let us consider a vertex $v\in\Vrt(\St_1(K)\otimes\ldots\otimes\St_n(K))$.  We can write this  as $v'+v_n$, where $v'\in\Vrt(\St(K_1)\otimes\cdots\otimes\St(K_{n-1}))$
and $v_n\in\Vrt(\St(K_n))$. We write the coordinates of the vertices as $(v_1,v_2)$, $(v'_1,v'_2)$ and $(v_{n1},v_{n2})$, respectively. We have
$v_1=v'_1+v_{n1}$, $v_2=v'_2+v_{n2}$. We shall need the following lemma.

\begin{lemma}\label{lem:magick}
For any four integers $x,y,z,w$ we have 
\[\max(x+y,z+w)=\min_{k\in\Z}\left(\max(x,z-k)+\max(y,w+k)\right).\]
\end{lemma}

\begin{proof}[Proof of Lemma~\ref{lem:magick}]
The direction `$\le$' is trivial. The equality is attained  at $k=z-x$.
\end{proof}

\smallskip
\emph{Continuation of the proof of Proposition~\ref{prop:ikjk}.}

Applying Lemma~\ref{lem:magick} to $v_1',v_2',v_{n1},v_{n2}-m$ and taking the minimum over all vertices $v$ we obtain
\begin{multline*}
J_K(m)=\min_{v\in\Vrt(\St(K_1)\otimes\ldots\otimes\St(K_n))}\max(v_1,v_2-m)=\\
\min_{v'\in\Vrt'}\min_{v_n\in\Vrt_n}\min_{k\in\Z}\left(\max(v'_1,v_2'-k)+\max(v_{n1},v_{n2}+k-m)\right),
\end{multline*}
where we denote $\Vrt'=\Vrt(\St(K_1)\otimes\cdots\otimes\St(K_{n-1}))$ and $\Vrt_n=\Vrt(\St(K_n))$.

\noindent The last expression is clearly $\min_{k\in\Z}J_{K'}(k)+J_{K_n}(m-k)$. By the induction assumption this is equal to
\[\min_{k\in\Z} I_{K'}(k+g')+I_{K_n}(m-k+g_n)=I_K(m+g),\]
where $g'=g(K')$ and $g_n=g(K_n)$ are the genera, and we use  the fact that $g=g'+g_n$.
\end{proof}

\section{Examples and applications}\label{sec:four}

\subsection{A certain curve of degree $6$}\label{sec:certain}

As described, for instance, in~\cite[Section 2.3, Table 1]{FLMN04}, there exists an algebraic curve of degree $6$ with two singular points, the
links of which are $K=T(4,5)$ and $K'=T(2,9)$.
The values of $I_K(m)$ for $m\in\{0,\ldots,11\}$ are $\{6,6,5,4,3,3,3,2,1,1,1,1\}$. The values of $I_{K'}(m)$ for $m\in\{0,\ldots,7\}$
are $\{4,4,3,3,2,2,1,1\}$. We readily get
\[I\diamond I'(1)=10,\ I\diamond I'(7)=6,\ I\diamond I'(13)=3,\ I\diamond I'(19)=1,\]
exactly as predicted by Theorem~\ref{thm:main}.

On the other hand, the computations in \cite{FLMN04} confirm Conjecture~\ref{conj:one} but we   sometimes have an inequality. For
example $k_6=5$, whereas Conjecture~\ref{conj:one} states $k_6\le 6$. This shows that Theorem~\ref{thm:main} is indeed more precise.

\subsection{Reformulations of Theorem~\ref{thm:main}}\label{sec:reform}
Theorem~\ref{thm:main} was formulated in a way that fits best with its theoretical underpinnings. In some applications, it is advantageous to reformulate the
result in terms of the function counting semigroup elements in the interval $[0,k]$. To this end, we introduce some notation.

Recall that for a semigroup $S\subset \Z_{\ge 0}$, the gap sequence of $G$ is $\Z_{\ge 0}\sm S$. We put $g=\#G$ and
for $m\ge 0$ we define
\begin{equation}\label{eq:rm-def}
R(m)=\#\{j\in S\colon j\in[0,m)\}.
\end{equation}
\begin{lemma}
For $m\ge 0$, $R(m)$ is related to the gap function $I(m)$ (see \eqref{eq:gapfunction})
by the following relation:
\begin{equation}\label{eq:rm}
R(m)=m-g+I(m).
\end{equation}
\end{lemma}
\begin{proof}
Let us consider an auxiliary function $K(m)=\#\{j\in[0,m): j\in G\}$. Then $K(m)=g-I(m)$. Now $R(m)+K(m)=m$, which completes the proof.
\end{proof}
We extend $R(m)$ by \eqref{eq:rm} for all $m\in\Z$. We remark that $R(m)=m-g$ for $m>\sup G$ and $R(m)=0$ for $m<0$. In particular,
$R$ is a non-negative, non-decreasing function.

We have the following result.
\begin{lemma}\label{lem:RandI}
Let $I_1,\dots,I_n$ be the gap functions corresponding to the semigroups $S_1,\ldots,S_n$. Let $g_1,\dots,g_n$ be given by $g_j=\#{\Z_{\ge 0}\sm S_j}$.
Let $R_1,\ldots,R_n$ be as in \eqref{eq:rm-def}. Then
\[R_1\diamond R_2\diamond \ldots\diamond R_n(m)=m-g+I_1\diamond \ldots \diamond I_n(m),\]
where $g=g_1+\ldots+g_n$.
\end{lemma}
\begin{proof}
To simplify the notation, we assume that $n=2$; the general case follows by  induction. We have
\begin{align*}
R_1\diamond R_2(m)=&\min_{k\in\Z}R_1(k)+R_2(m-k)=\\
&=\min_{k\in\Z}(k-g_1+I_1(k)+m-k-g_2+ I_2(m-k))=\\
&=m-g_1-g_2+I_1\diamond I_2(m).
\end{align*}
\end{proof}

Now we can reformulate Theorem~\ref{thm:main}:
\begin{theorem}\label{cor:main}
For any rational cuspidal curve of degree $d$ with singular points $z_1,\dots,z_n$, and for $R_1,\dots,R_n$ the functions as defined in \eqref{eq:rm-def}, one has that 
for any $j=\{-1,\ldots,d-2\}$,
\[R_1\diamond R_2\diamond \ldots \diamond R_n(jd+1)=\frac{(j+1)(j+2)}{2}.\]
\end{theorem}
This formulation follows from Theorem~\ref{thm:main} by an easy algebraic manipulation  together with the observation that by \eqref{eq:gofC}
and Lemma~\ref{lem:gaplem}, the quantity
$g$ from Lemma~\ref{lem:RandI} is given by $\frac{(d-1)(d-2)}{2}$.

The formula bears strong resemblance to \cite[Proposition 2]{FLMN04}, but in that article only the `$\ge$' part is proved and an equality 
in case $n=1$ is conjectured.

\begin{remark}\label{rem:import}
Observe that by definition
\[R_1\diamond \ldots \diamond R_n(k)=\min_{\substack{k_1,\ldots,k_n\in\Z\\ k_1+\ldots+k_n=k}}R_1(k_1)+\ldots+R_n(k_n).\]
Since for negative values $R_j(k)=0$ and $R_j$ is non-decreasing on $[0,\infty)$, the minimum will always be achieved for $k_1,\ldots,k_n\ge -1$.
\end{remark}

\subsection{Applications}
From Theorem~\ref{cor:main} we can deduce many general estimates for rational cuspidal curves. Throughout this subsection we shall be assuming
that $C$ has degree $d$, its singular points are $z_1,\ldots,z_n$, the semigroups are $S_1,\ldots,S_n$, and the corresponding $R$--functions
are $R_1,\ldots,R_n$. Moreover, we assume that the characteristic sequence of the singular point $z_i$ is
$(p_i;q_{i1},\ldots,q_{ik_i})$.  We order the singular points so that
that $p_1\ge p_2\ge\ldots\ge p_n$. 

We can immediately prove the result of Matsuoka--Sakai, \cite{MaSa89}, following the ideas in \cite[Section 3.5.1]{FLMN04}.

\begin{proposition}\label{prop:matsa}
We have $p_1>d/3$.
\end{proposition}
\begin{proof}
Suppose $3p_1\le d$. It follows that for any $j$, $3p_j\le d$. 
Let us choose $k_1,\ldots,k_n\ge -1$
such that $\sum k_j=d+1$.
For any $j$, the elements $0,p_j,2p_j,\ldots$ all belong to the $S_j$. The function $R_j(k_j)$ counts elements in $S_j$
strictly smaller than $k_j$, hence for any $\varepsilon>0$ we have
\[R_j(k_j)\ge 1+\intfrac{k_j-\varepsilon}{p_j}.\]
Using $3p_j\le d$ we rewrite this as $R_j(k_j)\ge 1+\intfrac{3k_j-3\varepsilon}{d}$.
Since $\varepsilon>0$ is arbitrary, setting $\delta_j=1$ if $d|3k_j$, and $0$ otherwise, we write
\[R_j(k_j)\ge 1+\intfrac{3k_j}{d}-\delta_j.\]
We get
\begin{equation}\label{eq:r1}
\sum_{j\colon d|3k_j} R_j(k_j)\ge \intfrac{\sum 3k_j}{d}.
\end{equation}
Using the fact that $\intfrac{a}{d}+\intfrac{b}{d}\ge \intfrac{a+b}{d}-1$ for any $a,b\in\Z$, we estimate the other terms: \begin{equation}\label{eq:r2}
\sum_{j\colon d\not\;|\,3k_j}R_j(k_j)\ge 1+\intfrac{3\sum k_j}{d}.
\end{equation}
Since $\sum k_j=d+1$, there must be at least one $j$ for which $d$ does not divide $3k_j$. Hence adding \eqref{eq:r1} to \eqref{eq:r2}
we obtain
\[
R_1(k_1)+\ldots+R_n(k_n)\ge 1+\intfrac{\sum_{j=1}^n 3k_j}{d}=1+\intfrac{3d+3}{d}=4.
\]
This
contradicts Theorem~\ref{cor:main} for $j=1$, and the contradiction concludes the proof.
\end{proof}

We also have the following simple result.
\begin{proposition}\label{prop:strong}
Suppose that $p_1>\frac{d+n-1}{2}$. Then $q_{11}<d+n-1$.
\end{proposition}
\begin{proof}
Suppose that $p_1>\frac{d+n-1}{2}$ and $q_{11}>d+n-1$. It follows that $R_1(d+n)=2$. But then we choose $k_1=d+n$, $k_2=\ldots=k_n=-1$
and we get $\sum_{j=1}^nR_j(k_j)=2$, hence
\[R_1\diamond R_2\diamond \ldots \diamond R_n(d+1)\le 2\]
contradicting Theorem~\ref{cor:main}.
\end{proof}

\subsection{Some examples and statistics}

We will now present some examples and statistics, where we compare our new criterion with the semicontinuity of the 
spectrum as used in \cite[Property $(SS_l)$]{FLMN04}
and the Orevkov criterion \cite[Corollary 2.2]{Orev02}. It will turn out that the semigroup distribution property is quite strong and   closely related
to the semicontinuity of the spectrum, but they are not the same. There are cases which pass one criterion and fail to another. Checking
the semigroup property is definitely a much faster task  than comparing spectra; refer to~\cite[Section~3.6]{FLMN06} for more examples.

\begin{example}\label{ex:twomilions} 
Among the 1,920,593 cuspidal singular points with Milnor number of the form $(d-1)(d-2)$ for $d$ ranging between $8$ and $64$,
there are only 481   that pass the semigroup distribution criterion, that is Theorem~\ref{thm:first}. All of these  pass the Orevkov criterion
$\ol{M}<3d-4$.  Of those 481, we compute that 475 satisfy the semicontinuity of the spectrum condition and  6 them fail the condition; 
these are: $(8; 28,45)$, $(12; 18, 49)$,
$(16; 56, 76, 85)$, $(24; 36, 78, 91)$, $(24; 84, 112, 125)$, $(36; 54, 114, 133)$.
\end{example}
\begin{remark}
The computations in Example~\ref{ex:twomilions} were made on a PC computer during one afternoon. Applying the spectrum criteria
for all these cases would take much longer. The computations for degrees between $12$ and $30$ is approximately $15$ times faster for semigroups; the difference seems to grow with
the degree. The reason is that even though the spectrum can be given explicitly from the characteristic sequence (see \cite{Sait}),
it is a set of fractional numbers and the algorithm is complicated.
\end{remark}

\begin{example}
There are $28$ cuspidal singular points with Milnor number equal to $110=(12-1)(12-2)$. We ask, which of
these singular points can possibly occur as a unique singular point on a degree $12$ rational curve?
We apply the semigroup distribution criterion.  Only 8 singular points pass the criterion, as is seen on Table~\ref{fig:12}.
\end{example}

\begin{table}[h]
\begin{tabular}{||l|c||l|c||l|c||}\hline \hline
(3;56) & fails at $j=1$ & (6;9,44) & fails at $j=1$ & (8;12,14,41) & fails at $j=3$ \\ \hline
 (4;6,101) & fails at $j=1$ & (6;10,75) & fails at $j=1$ & (8;12,18,33) & fails at $j=4$ \\ \hline
 (4;10,93) & fails at $j=1$ & (6;14,59) & fails at $j=2$ & (8;12,22,25) & \textbf{passes} \\ \hline
 (4;14,85) & fails at $j=1$ & (6;15,35) & fails at $j=2$ & (8;12,23) & \textbf{passes} \\ \hline
 (4;18,77) & fails at $j=1$ & (6;16,51) & fails at $j=2$ & (8;14,33) & fails at $j=1$ \\ \hline
 (4;22,69) & fails at $j=1$ & (6;20,35) & fails at $j=4$ & (9;12,23) & \textbf{passes} \\ \hline
 (4;26,61) & fails at $j=1$ & (6;21,26) & \textbf{passes} & (10;12,23) & \textbf{passes} \\ \hline
 (4;30,53) & fails at $j=1$ & (6;22,27) & \textbf{passes} & (11;12) & \textbf{passes} \\ \hline
 (4;34,45) & fails at $j=1$ & (6;23) & \textbf{passes}  & & \\ \hline
 (6;8,83) & fails at $j=1$ & (8;10,57) & fails at $j=2$  & & \\ \hline\hline
\end{tabular}\vskip.1in
\caption{Semigroup property for cuspidal singular points with Milnor number $12$. If a cuspidal singular point fails the
semigroup criterion, we indicate the first $j$ for which $I(12j+1)\neq \frac{(j-d+1)(j-d+2)}{2}$.}\label{fig:12}
\end{table}

Among the curves in Table~\ref{fig:12}, all those that are obstructed by the semigroup distribution, are also obstructed 
by the semicontinuity of the spectrum.  The spectrum also obstructs  the case of $(8;12,23)$.

\begin{example}
There are 2330 pairs $(a,b)$ of coprime integers, such that $(a-1)(b-1)$ is of form $(d-1)(d-2)$ for $d=5,\ldots,200$. 
Again we ask  if there exists a degree $d$ rational cuspidal curve having a single singular point with characteristic sequence $(a;b)$. Among these 2330 cases, precisely 302
satisfy the semigroup distribution property. Out of these 302 cases, only one, namely $(2;13)$, 
does not appear on the list from \cite{FLMN04};  see Section~\ref{subseccusp} for the list.
It is therefore very likely   that the semigroup distribution property alone is strong enough to obtain the classification of \cite{FLMN04}.
\end{example}

\begin{example}
In Table~\ref{fig:30} we present all the cuspidal points 
with Milnor number $(30-1)(30-2)$  that satisfy
the semicontinuity of the spectrum. Out of these, all
but the three ($(18;42,65)$, $(18;42,64,69)$ and $(18;42,63,48)$)
satisfy the semigroup property. All  three fail the semigroup property for $j=1$. 
In particular, for these three cases  the semigroup property obstructs the cases which pass the semicontinuity of the spectrum criterion.
\end{example}

\begin{table}[h]
\begin{tabular}{||l|l|l|l||}\hline
(15; 55, 69) & (18;42,64,69) & (20; 30, 59) & (25; 30, 59) \\\hline
(15; 57, 71) & (18;42,63,68) & (24; 30, 57, 62) & (27; 30, 59) \\\hline
(15;59) & (20; 30,55,64) & (24;30,58,63) & (28; 30,59) \\\hline
(18;42,65) & (20; 30,58,67) &  (24; 30,59) & (29; 30)\\\hline
\end{tabular}\vskip.1in
\caption{Cuspidal singular points with Milnor number $752$ satisfying
the semicontinuity of the spectrum criterion.}\label{fig:30}
\end{table}

\begin{example}
The configuration of five critical points $(2; 3)$, $(2; 3)$, $(2; 5)$, $(5; 7)$ and  $(5; 11)$ passes the semigroup, the spectrum and the
Orevkov criterion for a degree $10$ curve. In other words, none of the aforementioned criteria obstructs the existence of such curve.
We point out that it is conjectured (see \cite{Moe08,Pion}) that a rational cuspidal curve can have at most $4$ singular points. In other words,
these three criteria alone are insufficient to prove that conjecture.
\end{example}

\end{document}